\documentclass[11pt,a4paper]{article}
\usepackage{amssymb,amsmath,amsthm}

\newcommand\cF{{\mathcal F}}
\newcommand\cG{{\mathcal G}}

\newcommand\cP{{\mathcal P}}

\theoremstyle{plain}
\newtheorem{theorem}{Theorem}[section]
\newtheorem{lemma}[theorem]{Lemma}

\theoremstyle{definition}

\newtheorem{defn}[theorem]{Definition}

\newcommand\lref[1]{Lemma~\ref{lem:#1}}
\newcommand\tref[1]{Theorem~\ref{thm:#1}}
\newcommand\cref[1]{Corollary~\ref{cor:#1}}

\newcommand\dref[1]{Definition~\ref{def:#1}}

\title{A note on tilted Sperner families with patterns}

\author{D\'aniel Gerbner\thanks{MTA Alfr\'ed R\'enyi Institute of Mathematics, P.O.B. 127, Budapest H-1364, Hungary.  Email: gerbner.daniel@renyi.mta.hu Research supported by OTKA grant PD-109537.} \and M\'at\'e Vizer\thanks{MTA Alfr\'ed R\'enyi Institute of Mathematics, P.O.B. 127, Budapest H-1364, Hungary. Email: vizermate@gmail.com Research supported by OTKA grant SNN--116095.}}

\begin{document}
\maketitle

\begin{abstract}

Let $p$ and $q$ be two nonnegative integers with $p+q>0$ and $n>0$. We call $\cF \subset \cP([n])$ a \textit{(p,q)-tilted Sperner family with patterns} on [n] if there are no distinct $F,G \in \cF$ with:
$$(i) \ \ p|F \setminus G|=q|G \setminus F|, \ \textrm{and}$$
$$(ii) \ f > g  \ \textrm{for all} \ f \in F \setminus G \ \textrm{and} \ g \in G \setminus F.$$

\vspace{3mm}

E. Long in \cite{L} proved that the cardinality of a (1,2)-tilted Sperner family with patterns on $[n]$ is $$O(e^{120\sqrt{\log n}}\ \frac{2^n}{\sqrt{n}}).$$

We improve and generalize this result, and prove that the cardinality of every ($p,q$)-tilted Sperner family with patterns on [$n$] is $$O(\sqrt{\log n} \ \frac{2^n}{\sqrt{n}}).$$

\end{abstract}

\textit{Keywords: Sperner family, tilted Sperner family, permutation method}

\medskip

\section{Introduction}

A family $\cF$ of subsets of $[n]$ (where for $n>0$ we will use the $[n]$ notation for $\{1,2,...,n\}$ and $\cP([n])$ for the power set) is called a \textit{Sperner family} if $F \not \subset G$ for all distinct $F,G \in \cF$. A classic result in extremal combinatorics is Sperner's theorem \cite{S}, which states that the maximal cardinality of a Sperner family is $\binom{n}{\lfloor \frac{n}{2} \rfloor}$. This result has a huge impact on combinatorics and has many generalizations (see e.g. \cite{E}).

Recently Sperner's theorem played some role in the Polymath project to discover a new proof of the density Hales-Jewett theorem \cite{P}. Motivated by its role in the proof Kalai asked whether one can achieve 'Sperner-like theorems' for 'Sperner like families' \cite{LL}.

One direction to generalize the notion of Sperner families is the so called \textit{tilted Sperner families} (see \dref{1}). As written in \cite{LL}: Kalai noted that the 'no containment' condition can be
rephrased as follows: $\cF$ does not contain two sets $F$ and $G$ such that, in the unique subcube of $\cP([n])$
spanned by $F$ and $G$, the bottom point is $F$ and $G$ is the top point. He asked: what happens if we
forbid $F$ and $G$ to be at a different position in this subcube? In particular, he asked how large $\cF \subset \cP([n])$
can be if we forbid $F$ and $G$ to be at a ‘fixed ratio’ $p : q$ in this subcube. That is, we forbid $F$ to be
$p/(p+q)$ of the way up this subcube and $G$ to be $q/(p+q)$ of the way up this subcube. Equivalently we can say:

\begin{defn}
\label{def:1}
Let $p,q$ be two nonnegative integers. We call $\cF \subseteq \cP([n])$  a \textit{(p,q)-tilted Sperner family} if for all distinct $F,G \in \cF$ we have $$p|F \setminus G| \not = q|G \setminus F|.$$

\end{defn}

Note that we can restrict ourselves to coprime $p$ and $q$. Also note the a Sperner family is just a $(1,0)$-tilted Sperner family. In \cite{LL} Leader and Long proved the following theorem, which gives an asymptotically tight answer for the maximal cardinality of a \textit{(p,q)}-tilted Sperner family:

\begin{theorem}
Let $p,q$ be coprime nonnegative integers with $q \ge p$. Suppose $\cF \subset \cP([n])$ is a $(p,q)$-tilted Sperner family. Then
$$|\cF| \le (q-p+o(1)) \binom{n}{\lfloor \frac{n}{2} \rfloor}.$$

\end{theorem}

Note that up to the $o(1)$ term, this is the best possible, since the union of $p-q$ consecutive levels is a $(p,q)$-tilted Sperner family.

In \cite{L} Long started to investigate the cardinality of \textit{tilted Sperner families with patterns} (see \dref{2}), which was also asked by Kalai (\cite{Le}).

\begin{defn}
\label{def:2}
Let $p$ and $q$ be nonnegative integers with $p+q>0$. We call $\cF$ a \textit{(p,q)-tilted Sperner family with patterns}, if there are no distinct $F,G \in \cF$ with:

\vspace{3mm}

$(i)$ $p|F \setminus G|=q|G \setminus F|$, and

\vspace{2mm}

$(ii)$ $f > g$ for all $f \in F \setminus G$ and $g \in G \setminus F$.

\vspace{3mm}

\end{defn}

In \cite{L} he gave an upper bound on the cardinality of a (1,2)-tilted Sperner family with patterns:

\begin{theorem}(\cite{L}, Theorem 1.3)
\label{thm:Long}
Let $\cF \subset \cP([n])$ be a (1,2)-tilted Sperner family with patterns. Then
$$|\cF| \le O(e^{120\sqrt{\log n}} \ \frac{2^n}{\sqrt{n}}).$$

\end{theorem}

Actually in \cite{L} he gives a proof of a weaker result with the density Hales-Jewett theorem, and proves \tref{Long} with a randomized generalization of Katona's cycle method (see \cite{K}).

\vspace{2mm}

In this note we generalize and improve his result by applying another generalization of Katona's cycle method, the so called permutation method. We will apply the permutation method in a somewhat similar way like the authors of \cite{EFK} and prove the following:

\begin{theorem}
\label{thm:gv}

Let p and q be non negative integers with $p+q >0$ and let $\cF$ be a (p,q)-tilted Sperner family with patterns. Then

$$|\cF| \le O(\sqrt{\log n} \ \frac{2^n}{\sqrt{n}}).$$

\end{theorem}

The paper is organized as follows: in Section 2 we prove our main theorem and in Section 3 we pose some questions.

\section{Proof of \tref{gv}}

\begin{proof}

If either $p$ or $q$ is zero, then we get back the usual Sperner family for which we know that the statement is true. In the following we fix $p,q>0$ and furthermore we assume that $p\le q$. The proof works similarly in case $p >q$.

\subsection{The $(p,q)$-cut point}

First we introduce a notion that will have crucial role in the proof.

\begin{defn}

We say that $x \in [n]$ is a \textit{(p,q)-cut point of} $A \subseteq [n]$, if

\begin{equation}
\label{eq:pq}
0 \le \frac{n-x-|([n]\setminus[x])\cap A|}{q}-\frac{|A\cap[x]|}{p} < \frac{1}{p}.
\end{equation}

We remark that $x$ is a $(p,q)$-cut point means that $\frac{p}{q}$ times the number of points of $A$ less than $x$ is 'approximately' equal to the number of points not belonging to $A$ that are larger than $x$.

\end{defn}

\begin{lemma}
\label{lem:cutpoint}
Every $A \subseteq [n]$ has a (p,q)-cut point.

\end{lemma}

\begin{proof} Let us introduce the following functions: for $u \in \{0 \} \cup [n]$ and $A \subseteq [n]$ let

$$f(A,u):=\frac{|A \cap [u]|}{p} \ \ \  \textrm{and} \ \ g(A,u):=\frac{n-u-|([n]\setminus[u])\cap A|}{q},$$

with $|A \cap [0]|=0$. Observe that if $|A|\not = 0$, then we have

\begin{equation}
\label{eq:11}
 \ \ 0=f(A,0)< g(A,0)=\frac{n-|A|}{q}  \ \ \textrm{and} \ \
\frac{|A|}{p}=f(A,n) > g(A,n)=0.
\end{equation}

Also note that for all $i \in [n]$ if

\vspace{3mm}

$\bullet_1$ $i \in A$, then

$$f(A,i-1)+\frac{1}{p}=f(A,i) \ \ \textrm{and} \ \  g(A,i-1)=g(A,i)$$

$\bullet_2$ $i \not \in A$, then

$$f(A,i-1)=f(A,i) \ \ \textrm{and} \ \ g(A,i-1)-\frac{1}{q}=g(A,i).$$

By $\bullet_1$,$\bullet_2$ and $(\ref{eq:11})$ we have $f(A,0)<g(A,0)$ and going towards $n$, $f$ is increasing, $g$ is decreasing, but both of them changes with at most $\frac{1}{p}$ and we have $f(A,n)>g(A,n)$.

We are done with the proof of \lref{cutpoint}.

\end{proof}

\vspace{5mm}

\subsection{Using the permutation method}

\
Let us introduce two pieces of notation:

\vspace{3mm}

1) for all $F \in \cF$ choose a $(p,q)$-cut point $x_F$ (we can do it by \lref{cutpoint}), and let $$\cF_x:=\{F \in \cF: x=x_F\} \ \ \textrm{for} \ \ x \in [n],$$

\vspace{2mm}

2) for $x+k \le n$ let $j(x,k):= \lfloor \frac{p}{q}(n-x-k) \rfloor$.

\vspace{4mm}

Note that if $x$ is a $(p,q)$-cut point for $A \subseteq [n]$, then

$$|A\cap[x]|=j(x,|([n]\setminus[x])\cap A|).$$

\vspace{3mm}

In this section we will prove an upper bound on $|\cF_x|$ using the permutation method.

\vspace{1mm}

Let us consider the following permutation group of $[n]$: for any $x\in[n]$ let us denote by $S_x$ the symmetric group on $x$ elements, and let $\Pi_x:=S_x \times S_{n-x}$, the direct product of $S_x$ and $S_{n-x}$ (for definition of direct product of groups see e.g. \cite{La}). An element $(\pi_1,\pi_2)=\pi \in \Pi_x$ acts on $[n]$ the following way:

$$\pi(i) = \left\{ \begin{array}{lll}
\pi_1(i) & \textrm{if} \ \ i \le x, \\
\pi_2(i-x)+x &  \textrm{if} \ \ i>x.
\end{array} \right.$$

\vspace{2mm}

For $A \subseteq [n]$ and $\pi \in \Pi_x$ we will use the notation $\pi(A)$ for $\{\pi(a) : a \in A\}$.

\vspace{3mm}

Let us define the following families of sets for $x \in [n]$, $0 \le k \le n-x$ if $j(x,k)<x$:

$$C(x,k):=\{1,2,...,j(x,k), x+1, x+2,..., x+k\}.$$

\vspace{5mm}

Observe two things:

\vspace{3mm}

$\circ_1$ For any $x \in [n]$ and $r < q$ we have

$$|\{C(x,tq+r): 0 \le t \le \frac{n}{q}\} \cap \cF| \le 1$$

by the assumptions that $\cF$ is a ($p,q$)-tilted Sperner family with patterns and two such sets for different $t'$s are forbidden. Note here that $C(x,tq+r)$ does not even exist for some $t$. We also have that for all $\pi \in \Pi_x$

$$|\{\pi(C(x,tq+r)): 0 \le t \le \frac{n}{q}\} \cap \cF_x| \le 1.$$

Indeed, if $F$ and $G$ are both in this family, it is easy to calculate that $p|F \setminus G|=q|G \setminus F|$, and elements of $F \setminus G$ are smaller than $x$ while elements of $G \setminus F$ are larger than $x$.

\vspace{2mm}

$\circ_2$ For any $F\in \cF_x$ there are $k \le n-x$ and $\pi \in \Pi_x$ with

$$F=\pi(C(x,k)).$$

\vspace{3mm}

Now let us do the following computation: fix $x \in [n]$. Using $\circ_1$ we have the following

$$\sum_{\pi \in \Pi_x}\sum_{r=0}^{q-1}\sum_{t=0}^{\lfloor \frac{n}{q}\rfloor}|\pi(C(x,tq+r))\cap \cF_x| \le q(n-x)!x!.$$

After changing the order summations using $\circ_2$ we get

$$\sum_{F \in \cF_x}|F\cap[x]|!(x-|F\cap[x]|)!(|F\setminus[x]|)!(n-x-|F\setminus[x]|)! \le q(n-x)!x!,$$

\vspace{3mm}

and finally, dividing both sides by $(n-x)!x!$ we have

\begin{equation}
\label{eq:bin}
\sum_{F \in \cF_x}\frac{1}{\binom{x}{| F \cap [x]|}\binom{n-x}{|F\setminus [x]|}}\le q.
\end{equation}

\vspace{3mm}

Using the fact that $\binom{x}{i} \le 2^x/\sqrt{x}$, from (\ref{eq:bin}) we have that for all $x \in [n]$:

\begin{equation}
\label{eq:fx}
|\cF_x|\le O(\frac{2^n}{\sqrt{x(n-x)}}).
\end{equation}

\subsection{Finishing the proof of \tref{gv}}

We finish the proof of \tref{gv} by a standard application of the Chernoff-Hoeffding bound (\cite{C}, \cite{H}):

\vspace{3mm}

\textbf{Chernoff-Hoeffding bound:} Let $X_i$ be
independent random variables in the $[0,1]$ interval and let $$X(n):= \sum_{i=1}^n X_i.$$ Then for $t \leq \mathbb{E}[X(n)]$ we have $$\mathbb{P}(|X(n)-\mathbb{E}[X(n)]| \geq t) \leq 2 \exp ( - \frac{2t^2}{n}).$$

The next lemma is probably well known, however for the sake of completeness we present a proof here. Let

$$\cG:=\{G \subseteq [n]  : \textrm{there is } x \in [n] \ \textrm{with} \ \big{|}|[x] \cap G|-\frac{x}{2}\big{|} > \sqrt{n \log n}\}.$$

\vspace{3mm}

\begin{lemma}
\label{lem:cherno}

We have $$|\cG| \le O(\frac{2^n}{n}).$$

\end{lemma}

\begin{proof}

Note that $\cG=\cup_{x \in [n]} \cG_x$, where $$\cG_x:=\{G \in \cG : \Big{|}|[x] \cap G|-\frac{x}{2}\Big{|} > \sqrt{n \log n}\}.$$

Observe that

\begin{equation}
\label{eq:cher1}
|\cG_x| \frac{1}{2^n} \le \big{(} \sum_{y=0}^{\lfloor \frac{x}{2}-\sqrt{n\log n}\rfloor}\binom{x}{y} + \sum_{y=\lceil \frac{x}{2}+\sqrt{n\log n} \rceil}^{x} \binom{x}{y} \big{)}\frac{1}{2^x}
\end{equation}

Applying the Chernoff-Hoeffding bound on the right hand side of (\ref{eq:cher1}) with $t=\sqrt{n \log n}$ \newline $(\textrm{which is less than } \frac{n}{2} \textrm{ for } n \geq 10 )$ we have

\begin{equation}
\label{eq:cher}
|\cG_x| \frac{1}{2^n} \le 2 exp(-\frac{2 n \log n}{x}).
\end{equation}

Using $x \le n$ on the right hand side of (\ref{eq:cher}), we have

$$|\cG_x| \le O(\frac{2^{n}}{n^2}),$$

which easily implies the statement of the lemma.

\end{proof}

\vspace{4mm}

Let $\cF':=\cF \setminus \cG$. \vspace{3mm} \newline Using \lref{cherno} we prove that a $(p,q)$-cut point of any $F \in \cF'$ is in a $O(\sqrt{n \log n})$ neighborhood of $\frac{p}{p+q}n$.

\begin{lemma}
\label{lem:cher}

For $n \geq 2$ and all $F \in \cF'$ we have $$|x_F -\frac{p}{p+q}n| \le 8\sqrt{n \log n}.$$

\end{lemma}

\begin{proof}

By the fact that $F \in \cF'$ we have both

\begin{equation}
\label{eq:end1}
 \Big{|}|[x_F] \cap F|- \lfloor \frac{x_F}{2} \rfloor \Big{|} \le \sqrt{n \log n}
\end{equation}

and

\begin{equation}
\label{eq:end2}
\Big{|}|[n] \cap F|- \lfloor \frac{n}{2} \rfloor \Big{|} \le \sqrt{n \log n}.
\end{equation}

By (\ref{eq:end1}) and (\ref{eq:end2}) we have (loosing at most 1 in putting together two inequalities and using that $1 \leq \sqrt{n \log n}$ for $n \geq 2$.)

\begin{equation}
\label{eq:end3}
\Big{|}|([n] \setminus [x_F]) \cap F|-\lfloor \frac{n-x_F}{2} \rfloor\Big{|} \le 4\sqrt{n \log n}.
\end{equation}

However $x_F$ is a $(p,q)$-cut point for $F$, so by (\ref{eq:end1}), (\ref{eq:end2}) and (\ref{eq:end3}) we have

$$\Big{|}(n-x_F -\lfloor \frac{n - x_F}{2} \rfloor)\frac{1}{q}- \lfloor \frac{x_F}{2}\rfloor \frac{1}{p} \Big{|} \le 8\sqrt{n \log n},$$

and we are done with \lref{cher}.

\end{proof}

By (\ref{eq:fx}) and \lref{cher} we have $$|\cF'| \le O(\sqrt{n \log n} \ \frac{2^n}{n}),$$

and by \lref{cherno} we are done with the proof of \tref{gv}.

\end{proof}

\section{Concluding remarks}

\

We proved in \tref{gv} that the cardinality of a $(p,q)$-tilted Sperner family with patterns on $[n]$ is $O(\sqrt{\log n} \ \frac{2^n}{\sqrt{n}})$, however we do not have much better constructions than the ones in \cite{LL}. We conjecture that for different $p$ and $q$ the order of a maximal size $(p,q)$-tilted Sperner family with patterns on $[n]$ is $\Theta(\frac{2^n}{\sqrt{n}})$.

For $p=q$ we are not able to give really good constructions, we only know that the $(0,0)$-tilted Sperner family with patterns on $[n]$ (which we define just with property $(ii)$ in \dref{2}) is $O(\frac{2^n}{n})$, and we do not know what should be the right order.

It is worth mentioning that the whole topic from a more general viewpoint is investigated in the recent paper \cite{KL}.

\section{Acknowledgment}

The authors would like to thank Zheijang Normal University, China - where they started to work on this problem - for their hospitality. They are also indebted to the anonymous referees for providing insightful comments which increased the level of presentation of the paper.


\begin{thebibliography}{99}

\bibitem{C} H. Chernoff: A measure of asymptotic efficiency for tests of a hypothesis based on the sum of observations, Ann. Math. Stat. 23, pp. 493--507, 1952.

\bibitem{E} K. Engel: Sperner Theory, Cambridge University Press, 1997.

\bibitem{EFK} Peter L. Erd\H{o}s, Z. F\"uredi and G.O.H. Katona: Two-part and k-Sperner families: new proofs using permutations, SIAM J. Discrete Math. 19, pp. 489--500, 2005.
    
\bibitem{H} W. Hoeffding. Probability inequalities for sums of bounded random variables, J. Am. Stat. Assoc. 58, pp. 13--30, 1963.


\bibitem{K} G.O.H. Katona: A simple proof of the Erd\H{o}s –- Chao Ko –- Rado theorem, J. Comb. Theory B 13, pp. 183--184, 1972.

\bibitem{KL} I. Karpas, E. Long: Set families with a forbidden pattern,  arXiv:1510.05134.

\bibitem{La} S. Lang: Undergraduate Algebra (3rd ed.), Berlin, New York: Springer-Verlag, 2005.

\bibitem{LL} I. Leader, E. Long: Tilted Sperner families, Disc. Appl. Math., 163, part 2, pp. 194--198, 2014.

\bibitem{Le} E. Long, email communication.

\bibitem{L} E. Long: Forbidding intersection patterns between layers of the cube, J. Comb. Theory A 134, pp. 103--120, 2015.

\bibitem{P} D.H.J. Polymath: A new proof of the density Hales-Jewett theorem, Ann. Math. 175, pp. 1283--1327, 2012.

\bibitem{S} E. Sperner: Ein Satz \"uber Untermengen einer endlichen Menge. Math. Z., 27, pp. 544--548, 1928.

\end{thebibliography}
\end{document}